\documentclass[a4paper,11pt,oneside]{amsart}
\usepackage[left=2.7cm,right=2.7cm,top=3.5cm,bottom=3cm]{geometry}

\usepackage[all]{xy}
\usepackage{amscd}
\usepackage{amsfonts}
\usepackage{amssymb}
\usepackage{amsmath}
\usepackage{latexsym}
\usepackage[usenames,dvipsnames]{color}

\usepackage[notref,notcite]{showkeys}

\numberwithin{equation}{section}

\newtheorem{thm}{Theorem}[section]
\newtheorem{cor}[thm]{Corollary}
\newtheorem{lem}[thm]{Lemma}
\newtheorem{prop}[thm]{Proposition}

\theoremstyle{definition}
\newtheorem{df}[thm]{Definition}
\newtheorem{rem}[thm]{Remark}
\newtheorem{rems}[thm]{Remarks}
\newtheorem{ass}[thm]{Assumption}

\newcommand{\N}{\mathbb{N}}
\newcommand{\Z}{\mathbb{Z}}

\newcommand{\F}{\mathbb{F}}
\newcommand{\G}{\Gamma}
\newcommand{\g}{\gamma}

\renewcommand{\L}{\Lambda}

\newcommand{\Gal}{\textrm{Gal}}
\def\Ker{\mathrm{Ker}\,}
\def\Coker{\mathrm{Coker}\,}

\newcommand{\calf}{\mathcal F}
\newcommand{\calg}{\mathcal G}
\newcommand{\calh}{\mathcal H}

\newcommand{\call}{\mathcal L}

\newcommand{\cals}{\mathcal S}

\newcommand{\calx}{\mathcal X}

\newcommand{\pr}{\mathfrak p}

\newcommand{\sri}{\twoheadrightarrow}
\newcommand{\iri}{\hookrightarrow}

\newcommand{\wt}{\widetilde}

\renewcommand{\geq}{\geqslant}

\newcommand{\dl}[1]{\lim_{\buildrel \longrightarrow\over{#1}}}
\newcommand{\il}[1]{\lim_{\buildrel \longleftarrow\over{#1}}}

\title{Characteristic ideals and Selmer groups}

\author[A. Bandini] {Andrea Bandini}
\address{Dipartimento di Matematica e Informatica, Universit\`a degli Studi di Parma\\
Parco Area delle Scienze, 53/A - 43124 Parma (PR), Italy}
\email{andrea.bandini@unipr.it}

\author[F. Bars] {Francesc Bars}
\address{Departament Matem\`atiques, Edif. C, Universitat Aut\`onoma de Barcelona\\
08193 Bellaterra, Catalonia}
\email{francesc@mat.uab.cat}
\thanks{F. Bars supported by MTM2013-40680-P}

\author[I. Longhi] {Ignazio Longhi}
\address{Department of Mathematical Sciences,
Xi'an Jiaotong-Liverpool University \\
111 Ren Ai Road,
Dushu Lake Higher Education Town\\
Suzhou Industrial Park, Suzhou, Jiangsu\\
215123, China} \email{Ignazio.Longhi@xjtlu.edu.cn}



\begin{document}

\begin{abstract}
Let $A$ be an abelian variety defined over a global field $F$ of positive characteristic $p$ and let $\calf/F$ be a
$\Z_p^{\N}$-extension, unramified outside a finite set of places of $F$. Assuming that all ramified places are
totally ramified, we define a pro-characteristic ideal associated to the Pontrjagin dual of the $p$-primary Selmer group
of $A$, in order to formulate an Iwasawa Main Conjecture for the non-noetherian commutative Iwasawa algebra
$\Z_p[[\Gal(\calf/F)]]$ (which we also prove for a constant abelian variety). To do this we first show the relation
between the characteristic ideals of duals of Selmer groups for a $\Z_p^d$-extension $\calf_d/F$ and for any
$\Z_p^{d-1}$-extension contained in $\calf_d\,$, and then use a limit process.
\end{abstract}

\maketitle

\section{Introduction}
Let $F$ be a global function field of characteristic $p$ and $\calf/F$ a $\Z_p^\N$-extension unramified outside
a finite set of places. We take an abelian variety $A$ defined over $F$ and let $S_A$ be a finite set of
places of $F$ containing exactly the ramified primes and the primes of bad reduction for $A$.
For any extension $v$ of some place of $F$ to the algebraic closure $\overline{F}$ and for any finite extension
$E/F$, we denote by $E_v$ the completion of $E$ with respect to $v$ and, if $\call/F$ is infinite,
we put $\call_v:=\cup E_v$, where the union is taken over all finite subextensions of $\call$.
We define the $p$-part of the {\em Selmer group} of $A$ over $E$ as
\[ Sel(E):=Sel_A(E)_p:=Ker\left\{ H^1_{fl}(X_E,A[p^\infty])\longrightarrow \prod_v H^1_{fl}(X_{E_v},A)[p^\infty] \right\} \]
(where the map is the product of the natural restrictions at all places $v$ of $E$, $X_E:=Spec(E)$ and $H^1_{fl}$ denotes
flat cohomology). For infinite algebraic extensions we define the Selmer groups by taking direct limits on all the finite subextensions.
For any algebraic extension $K/F$, let $\cals(K)$ denote the Pontrjagin dual of $Sel(K)$
(other Pontrjagin duals will be indicated by the symbol $\,^\vee\,$).

For any infinite Galois extension $\call/F$, let $\L(\call):=\Z_p[[\Gal(\call/F)]]$ be the associated Iwasawa algebra:
we recall that, if $\Gal(\call/F)\simeq\Z_p^d\,$, then $\L(\call)\simeq\Z_p[[t_1,..,t_d]]$ is a Krull domain. We are interested
in the definition of the analogue of a characteristic ideal in $\L(\calf)$ for $\cals(\calf)$
(a similar result providing a {\em pro-characteristic ideal} for the Iwasawa module of class groups
is described in \cite{BBL2}).

If $R$ is a noetherian Krull domain and $M$ a finitely generated torsion $R$-module, the structure
theorem for $M$ provides an exact sequence
\begin{equation} \label{e:1} 0 \longrightarrow P \longrightarrow M \longrightarrow \bigoplus_{i=1}^n R/\pr_i^{e_i} R
\longrightarrow Q \longrightarrow 0 \end{equation}
where the $\pr_i$'s are height 1 prime ideals of $R$ and $P$ and $Q$ are pseudo-null $R$-modules (i.e., torsion modules
with annihilator of height at least 2). With this sequence one defines the {\em characteristic ideal} of $M$ as
\[ Ch_R(M):= \prod_{i=1}^n \pr_i^{e_i} \]
(if $M$ is not torsion, we put $Ch_R(M)=0$ and if $M$ is pseudo-null, then $Ch_R(M)=(1)\,$).

\noindent In commutative Iwasawa theory characteristic ideals provide the algebraic counterpart for the $p$-adic $L$-functions
associated to Iwasawa modules (such as duals of Selmer groups). We fix a $\Z_p$-basis $\{\gamma_i\}_{i\in\N}$ for $\G$
and for any $d\geq 0$ we let $\calf_d\subset\calf$  be the fixed field of $\{\gamma_i\}_{i>d}\,$. The correspondence
$\gamma_i-1 \leftrightarrow t_i$ provides an isomorphism between the subring $\L(\calf_d)$ of $\L(\calf)$
and $\Z_p[[t_1,\dots, t_d]]\,$: by a slight abuse of notation, the two shall be identified in this paper. In particular, for any
$d\geq 1$ we have $\L(\calf_d)=\L(\calf_{d-1})[[t_d]]$ (we remark that the isomorphism is uniquely determined once
the $\g_i$ have been fixed, but we allow complete freedom in their initial choice). Our goal is to define an element
in $\L(\calf)$ associated to $\cals(\calf)$ via a limit of the $Ch_{\L(\calf_d)}(\cals(\calf_d))$.

Let $\pi^d_{d-1}$ be the canonical projection $\L(\calf_d)\rightarrow\L(\calf_{d-1})$ with kernel
$I^d_{d-1}=(t_d)$ and put $\G^d_{d-1}:=\Gal(\calf_d/\calf_{d-1})$. In order to define a pro-characteristic ideal for
the non-noetherian Iwasawa algebra $\L(\calf)$, we need to study the relation between
$\pi^{d}_{d-1}(Ch_{\L(\calf_d)}(S(\calf_d)))$ and $Ch_{\L(\calf_{d-1})}(S(\calf_{d-1}))$. A general
technique to deal with this type of descent and ensure that the limit does not depend on the filtration has been
described in \cite[Theorem 2.13]{BBL2}. That theorem is based on a generalization of some results of \cite[Section 3]{Oc}
(which directly apply to our algebras $\L(\calf_d)$, even without the generalization to Krull domains provided in \cite{BBL2})
and can be applied to the $\L(\calf)$-module $\cals(\calf)=\displaystyle{\il{d} \cals(\calf_d)}$.
In our setting \cite[Theorem 2.13]{BBL2} reads as follows

\begin{thm}\label{IntroThmBBL2}
If, for any $d \gg 1$, \begin{itemize}
\item[{\bf 1.}] the $t_d$-torsion submodule of $\cals(\calf_d)$ is a pseudo-null $\L(\calf_{d-1})$-module, i.e.,
\[ Ch_{\L(\calf_{d-1})}(\cals(\calf_d)_{t_d})=Ch_{\L(\calf_{d-1})}(\cals(\calf_d)^{\G^d_{d-1}})=(1)\ ;\]
\item[{\bf 2.}] $Ch_{\L(\calf_{d-1})}(\cals(\calf_d)/{t_d})=Ch_{\L(\calf_{d-1})}(\cals(\calf_d)/I^d_{d-1})
\subseteq Ch_{\L(\calf_{d-1})}(\cals(\calf_{d-1}))$,
\end{itemize}
then we can define a {\em pro-characteristic ideal} for $\cals(\calf)$ as
\[ \wt{Ch}_{\L(\calf)}(\cals(\calf)):= \il{d}\, (\pi^{\L(\calf)}_{\L(\calf_d)})^{-1}(Ch_{\L(\calf_d)}(\cals(\calf_d))\in \L(\calf) \]
(where $\pi^{\L(\calf)}_{\L(\calf_d)}\,: \L(\calf)\rightarrow \L(\calf_d)$ is the natural projection).
\end{thm}

In Section 2 we show that if $\cals(\calf_e)$ is $\L(\calf_e)$-torsion, then $\cals(\calf_d)$ is
$\L(\calf_d)$-torsion for any $d\geq e$ and use \cite[Proposition 2.10]{BBL2} to provide a general relation
\begin{equation}\label{DescentEq}
Ch_{\L(\calf_{d-1})}(\cals(\calf_d)^{\G^d_{d-1}})\cdot\pi^d_{d-1}(Ch_{\L(\calf_d)}(\cals(\calf_d))) =
Ch_{\L(\calf_{d-1})} (\cals(\calf_{d-1})) \cdot \theta^d_{d-1}
\end{equation}
(see \eqref{CharIdSel2} where the extra factor $\theta^d_{d-1}$ is more explicit). Then we move to the
totally ramified setting, i.e., extensions in which all ramified primes are assumed to be totally ramified.
An example are the extensions obtained from $F$ by adding the
$\mathfrak{a}^n$-torsion points of a normalized rank 1 Drinfeld module over $F$. In this setting, adapting some
techniques and results of K.-S. Tan (\cite{Tan11}), we check the hypotheses of Theorem \ref{IntroThmBBL2}
using equation \eqref{DescentEq}, and obtain (see Corollary \ref{CorTan} and Definition \ref{CharIdSel})

\begin{thm}\label{IntroThm}
Assume all ramified primes in $\calf/F$ are totally ramified, then for any $d \gg 0$
\[ \pi^d_{d-1} (Ch_{\L(\calf_d)}(\cals(\calf_d))) =
(p^\nu)\cdot Ch_{\L(\calf_{d-1})}(\cals(\calf_{d-1}))  \]
for some $\nu\geq 0$ (if no unramified place $v\in S_A$ splits completely in $\calf$, then $\nu=0$),
and the pro-characteristic ideal
\[ \wt{Ch}_{\L(\calf)}(\cals(\calf)):= \il{d}\, (\pi^{\L(\calf)}_{\L(\calf_d)})^{-1}(Ch_{\L(\calf_d)}(\cals(\calf_d)) \]
is well defined.
\end{thm}

As an application, we use a deep result of Lai - Longhi - Tan - Trihan \cite{LLTT} to show that such limit is the algebraic counterpart
of a $p$-adic $L$-function in the non-noetherian Iwasawa Main Conjecture for constant abelian varieties
(see Theorem \ref{FinalThm}).

\section{General $\Z_p$-descent for Selmer groups}\label{SelGrSec}
To be able to define characteristic ideals we need the following

\begin{thm}\label{SelStruct}
Assume that $A$ has good ordinary or split multiplicative reduction at all primes of the finite set $S_A$. Then,
for any $d$ and any $\Z_p^d$-extension $\call/F$ contained in $\calf$, the group $\cals(\call)$ is a finitely
generated $\L(\call)$-module.
\end{thm}

\begin{proof}
In this form the theorem is due to Tan (\cite{Tan}). See also \cite[Section 2]{BBL} and the references there.
\end{proof}

\noindent In order to obtain a nontrivial relation between the characteristic ideals, we
need something more than just Theorem \ref{SelStruct}, so we make the following

\begin{ass}\label{AssSel} There exists an $e>0$ such that $\cals(\calf_e)$ is a torsion $\L(\calf_e)$-module.
\end{ass}

\begin{rems}
\begin{itemize}  \item[{}]
\item[{\bf 1.}] This hypothesis is satisfied in many cases: for example when $\calf_e$ contains the arithmetic $\Z_p$-extension of
$F$ (proof in \cite[Theorem 2]{Tan11}, extending \cite[Theorem 1.7]{OT}) or when $Sel(F)$ is finite and $A$ has good
ordinary reduction at all places which ramify in $\calf_e/F$ (easy consequence of \cite[Theorem 4]{Tan}).
\item[{\bf 2.}] Our goal is an equation relating $\pi^d_{d-1}(Ch_{\L(\calf_d)}\cals(\calf_d))$ and the characteristic ideal of
$\cals(\calf_{d-1})$. If the above assumption is not satisfied for any $e$, then all characteristic ideals
are 0 and there is nothing to prove.
\end{itemize}
\end{rems}

In this section we also assume that none of the ramified prime has trivial decomposition group in $\Gal(\calf_1/F)$.
In Section \ref{pramSel} we shall work in extensions in which ramified places are totally ramified, so this assumption
will be automatically verified. Anyway this is not restrictive in general because of the following

\begin{lem} \label{RamLemma}
If $d>2$, one can always find a $\Z_p$-subextension $\calf_1/F$ of $\calf_d/F$ in which none of the ramified places
splits completely.
\end{lem}

\begin{proof}
See \cite[Lemma 3.1]{BBL2}
\end{proof}

Consider the diagram
\begin{equation} \label{DiagSel} \xymatrix{
Sel(\calf_{d-1}) \ar@{^(->}[r] \ar[d]^{a^d_{d-1}} & H^1_{fl}(\calx_{d-1},A[p^\infty]) \ar@{->>}[r] \ar[d]^{b^d_{d-1}} &
\calg(\calx_{d-1}) \ar[d]\ar[d]^{c^d_{d-1}} \\
Sel(\calf_d)^{\G^d_{d-1}} \ar@{^(->}[r] & H^1_{fl}(\calx_d,A[p^\infty])^{\G^d_{d-1}} \ar[r]  & \calg(\calx_d)^{\G^d_{d-1}} }
\end{equation}
where $\calx_d:=Spec(\calf_d)$ and $\calg(\calx_d)$ is the image of the product of the restriction maps
\[ H^1_{fl}(\calx_d,A[p^\infty])\longrightarrow \prod_w H^1_{fl}(\calx_{d,w},A)[p^\infty] \ , \]
with $w$ running over all places of $\calf_d$.

\begin{lem} \label{KercTor}
For any $d \geq 2$, the Pontrjagin dual of $\Ker c^d_{d-1}$ is a finitely generated torsion $\L(\calf_{d-1})$-module.
\end{lem}

\begin{proof}
For any place $v$ of $F$ we fix an extension to $\calf$, which by a slight abuse of notation we still denote by $v$,
so that the set of places of $\calf_d$ above $v$ will be the Galois orbit $\Gal(\calf_d/F)\cdot v$.
For any field $L$ let $\mathcal{P}_L$ be the set of places of $L$. We have an obvious injection
\begin{equation} \label{e:kerckerh}
\Ker c^d_{d-1} \iri \prod_{u\in\mathcal{P}_{\calf_{d-1}}} \Ker \left\{ H_{fl}^1(\calx_{d-1,u},A)[p^\infty]
\longrightarrow \prod_{w\mid u} H_{fl}^1(\calx_{d,w},A)[p^\infty] \right\}
\end{equation}
(the map is the product of the natural restrictions $r_w\,$). By the Hochschild-Serre spectral sequence, we get
\begin{equation}\label{kerrwh1}
\Ker r_w \simeq  H^1(\G^d_{d-1,w},A(\calf_{d,w}))[p^\infty]
\end{equation}
where $\G^d_{d-1,w}$ is the decomposition group of $w$ in $\G^d_{d-1}\,$. As $w$ varies among the places dividing $u$,
those kernels are isomorphic, so fixing primes $w\mid u\mid v$, we can regard $\Ker \prod_{w\mid u} r_w$ as a submodule of
\[ \calh_v(\calf_d):=\prod_{u\in\Gal(\calf_{d-1}/F)\cdot v} H^1(\G^d_{d-1,w},A(\calf_{d,w}))[p^\infty] \ .\]
Therefore $\calh_v(\calf_d)=0$ for all primes which totally split in $\calf_d/\calf_{d-1}$ and, from now on, we
only consider inert or ramified places.

\noindent Let $\Lambda(\calf_{i,v}):=\Z_p[[\Gal(\calf_{i,v}/F_v)]]$ be the Iwasawa algebra associated to the decomposition group
of $v$ and note that each $\Ker r_w$ is a $\L(\calf_{d-1,v})$-module.
Moreover, we get an action of $\Gal(\calf_d/F)$ on $\calh_v(\calf_d)$ by permutation of the primes and an isomorphism
\begin{equation} \label{e:lemmatan}
\calh_v(\calf_d)\simeq\L(\calf_{d-1})\otimes_{\L(\calf_{d-1,v})} H^1(\G^d_{d-1,w},A(\calf_{d,w}))[p^\infty] \end{equation}
for some fixed $w$ (provided that $H^1(\G^d_{d-1,w},A(\calf_{d,w})[p^\infty]$ is finitely generated over $\L(\calf_{d-1,v})$,
more details can be found in \cite[Section 3]{Tan11}).

\noindent First assume that the place $v$ is unramified in $\calf_d/F$ (hence inert in $\calf_d/\calf_{d-1}\,$).
Then $\calf_{d-1,v}=F_v\neq\calf_{d,v}$ and one has, by \cite[Proposition I.3.8]{Mi},
\[ H^1(\G^d_{d-1,w},A(\calf_{d,w}) \simeq H^1(\G^d_{d-1,w},\pi_0(\mathcal{A}_{0,v}))\ , \]
where $\mathcal{A}_{0,v}$ is the closed fiber of the N\'eron model of $A$ over $F_v$ and $\pi_0(\mathcal{A}_{0,v})$ its set
of connected components. It follows that $H^1(\G^d_{d-1,w},A(\calf_{d,w})[p^\infty]$ is trivial when $v$ does not lie above
$S_A$ and that it is finite of order bounded by (the $p$-part of) $|\pi_0(\mathcal{A}_{0,v})|$ for the unramified places
of bad reduction. Hence
\[ \xymatrix { \Ker c^d_{d-1} \ar@{^(->}[r] & \displaystyle{\bigoplus_{v\in S_A}\, \calh_v(\calf_d)} } \]
and, by \eqref{e:lemmatan}, $\calh_v(\calf_d)^\vee$ is a finitely generated torsion
$\L(\calf_{d-1})$-module for unramified $v$.

\noindent For the ramified case the exact sequence
\[ \xymatrix{ A(\calf_{d,w})[p]\, \ar@{^(->}[r]  & A(\calf_{d,w}) \ar@{->>}[r]^{p}  & pA(\calf_{d,w}) } \]
yields a surjection
\[ \xymatrix{ H^1(\G^d_{d-1,v},A(\calf_{d,w})[p]) \ar@{->>}[r]  & H^1(\G^d_{d-1,v},A(\calf_{d,w}))[p] \ .} \]
The first module is obviously finite, so $H^1(\G^d_{d-1,v},A(\calf_{d,w}))[p]$ is finite as well:
this implies that $H^1(\G^d_{d-1,w},A(\calf_{d,w})[p^\infty]^\vee$ has finite $\Z_p$-rank. As a finitely generated $\Z_p$-module,
$H^1(\G^d_{d-1,w},A(\calf_{d,w})[p^\infty]^\vee$ must be $\Z_p[[\G_{d-1,v}]]$-torsion (because of our choice of $\calf_1/F$)
and \eqref{e:lemmatan} shows once again that $\calh_v(\calf_d)^\vee$ is finitely generated and torsion over $\L(\calf_{d-1})$.
\end{proof}

\begin{rem} One can go deeper in the details and compute those kernels according to the reduction of $A$ at $v$ and the
behaviour of $v$ in $\calf_d/F$. We will do this in Section \ref{pramSel} but only for the particular case of a totally ramified
extension (with the statement of a Main Conjecture as a final goal). See \cite{Tan11} for a more general analysis.
\end{rem}

The following proposition provides a crucial step towards equation \eqref{DescentEq}
(in particular it takes care of hypothesis {\bf 2} of Theorem \ref{IntroThmBBL2}).

\begin{prop}\label{FinGenSel}
Let $e$ be as in Assumption \ref{AssSel}. For any $d> e$, the module $\cals(\calf_d)/I^d_{d-1}$ is a finitely
generated torsion $\L(\calf_{d-1})$-module and $\cals(\calf_d)$ is a finitely generated torsion $\L(\calf_d)$-module.
Moreover, if $d>\max\{2,e\}$,
\[ Ch_{\L(\calf_{d-1})}(\cals(\calf_d)/I^d_{d-1}) = Ch_{\L(\calf_{d-1})}(\cals(\calf_{d-1})) \cdot
Ch_{\L(\calf_{d-1})} ((\Coker a^d_{d-1})^\vee)\ . \]
\end{prop}

\begin{proof}
It suffices to prove the first statement for $d=e+1$, since then a standard argument (detailed e.g. in
\cite[page 207]{Gr1}) shows that $\cals(\calf_{e+1})$ is $\L(\calf_{e+1})$-torsion and we can iterate
the process. From diagram \eqref{DiagSel} one gets a sequence
\begin{equation}\label{DualSeqSel} \xymatrix{ (\Coker a^{e+1}_e)^\vee\, \ar@{^(->}[r]  &
(Sel(\calf_{e+1})^{\G^{e+1}_e})^\vee \longrightarrow \cals(\calf_e) \ar@{->>}[r]  & (\Ker a^{e+1}_e)^\vee \ . }
\end{equation}
By the Hochschild-Serre spectral sequence, it follows
$$\xymatrix{ \Coker b^{e+1}_e\, \ar@{^(->}[r]  & H^2(\G^{e+1}_e, A[p^\infty](\calf_{e+1}))=0 }$$
(because $\G^{e+1}_e$ has $p$-cohomological dimension 1). Therefore there is a surjective map
\[ \xymatrix{ \Ker c^{e+1}_e \ar@{->>}[r]  & \Coker a^{e+1}_e} \]
and, by Lemma \ref{KercTor},
$(\Coker a^{e+1}_e)^\vee$ is $\L(\calf_e)$-torsion. Hence Assumption \ref{AssSel} and sequence \eqref{DualSeqSel} yield that
\[ (Sel(\calf_{e+1})^{\G^{e+1}_e})^\vee \simeq \cals(\calf_{e+1})/I^{e+1}_e  \]
is $\L(\calf_e)$-torsion. To conclude note that (for any $d$) the duals of
\[ \Ker a^d_{d-1} \iri \Ker b^d_{d-1} \simeq H^1(\G^d_{d-1}, A[p^\infty](\calf_d))\simeq
A[p^\infty](\calf_d)/I^d_{d-1}  \]
are finitely generated $\Z_p$-modules (hence pseudo-null over $\L(\calf_{d-1})$ for any $d\geq 3$). Taking characteristic ideals
in the sequence \eqref{DualSeqSel}, for large enough $d$, one finds
\[ Ch_{\L(\calf_{d-1})}(\cals(\calf_d)/I^d_{d-1}) = Ch_{\L(\calf_{d-1})} (\cals(\calf_{d-1})) \cdot
Ch_{\L(\calf_{d-1})} ((\Coker a^d_{d-1})^\vee)\ . \]
\end{proof}

\noindent Whenever $\cals(\calf_d)$ is a finitely generated torsion $\L(\calf_d)$-module, \cite[Proposition 2.10]{BBL2} yields
\begin{equation}\label{CharIdSel1}
Ch_{\L(\calf_{d-1})}(\cals(\calf_d)^{\G^d_{d-1}})\cdot \pi^d_{d-1}(Ch_{\L(\calf_d)}(\cals(\calf_d))) =
Ch_{\L(\calf_{d-1})}(\cals(\calf_d)/I^d_{d-1}) \ .\end{equation}

\noindent If $d>\max\{2,e\}$, equation \eqref{CharIdSel1} turns into
\begin{equation}\label{CharIdSel2}
Ch_{\L(\calf_{d-1})}(\cals(\calf_d)^{\G^d_{d-1}})\pi^d_{d-1}(Ch_{\L(\calf_d)}(\cals(\calf_d))) =
Ch_{\L(\calf_{d-1})} (\cals(\calf_{d-1})) Ch_{\L(\calf_{d-1})}((\Coker a^d_{d-1})^\vee) \, .
\end{equation}

\noindent Therefore, whenever we can prove that $\cals(\calf_d)^{\G^d_{d-1}}$ is a pseudo-null $\L(\calf_{d-1})$-module (i.e.,
hypothesis {\bf 1} of Theorem \ref{IntroThmBBL2}), we immediately get
\begin{equation}\label{CharIdSel3}
\pi^d_{d-1}(Ch_{\L(\calf_d)}(\cals(\calf_d))) \subseteq Ch_{\L(\calf_{d-1})} (\cals(\calf_{d-1}))
\end{equation}
and Theorem \ref{IntroThmBBL2} will provide the definition of the pro-characteristic ideal for
$\cals(\calf)$ in $\L(\calf)$ we were looking for.

\section{$\Z_p$-descent for totally ramified extensions} \label{pramSel}
The main examples we have in mind are extensions satisfying the following

\begin{ass}\label{AssRam} The (finitely many) ramified places of $\calf/F$ are totally ramified. \end{ass}

\noindent In what follows an extension satisfying this assumption will be called a {\it totally ramified extension}.
A prototypical example is the $\mathfrak{a}$-cyclotomic extension of $\F_q(T)$ generated by the $\mathfrak{a}$-torsion of the Carlitz
module ($\mathfrak{a}$ an ideal of $\F_q[T]$, see e.g. \cite[Chapter 12]{Ro}).
As usual in Iwasawa theory over number fields, most of the proofs will work (or can be adapted) simply assuming that ramified primes
are totally ramified in $\calf/\calf_e$ for some $e\geq 0$, but, in the function field setting, one would need
some extra hypothesis on the behaviour of these places in $\calf_e/F$ (as we have seen with Lemma \ref{RamLemma}, note
that in totally ramified extensions any $\Z_p$-subextension can play the role of $\calf_1\,$).

\begin{thm}\label{CharIdCoker}
Assume $\calf/F$ is a totally ramified extension, then, for $d\gg 0$, there exists a
$\nu\geq 0$ such that
\[ Ch_{\L(\calf_{d-1})} ((\Coker\,a^d_{d-1})^\vee)\ = (p^\nu) \ . \]
Moreover if no unramified place $v\in S_A$ splits completely in $\calf$, then, for $d \gg 0$, we have
\[ Ch_{\L(\calf_{d-1})} ((\Coker\,a^d_{d-1})^\vee)\ = (1)\ . \]
\end{thm}

\begin{proof}
The proof of Proposition \ref{FinGenSel} shows that the $\L(\calf_{d-1})$-modules $(\Coker\,a^d_{d-1})^\vee$ and
$(\Ker c^d_{d-1})^\vee$ are pseudo-isomorphic for $d\geq 3$. Moreover, by the proof of Lemma \ref{KercTor},
we know that $(\Ker c^d_{d-1})^\vee$ is a quotient of $\displaystyle{\bigoplus_{v\in S_A} \calh_v(\calf_d)^\vee}\,$.
By equation \eqref{e:lemmatan}, we have
\begin{equation}\label{TanChId}
Ch_{\L(\calf_{d-1})}( \calh_v(\calf_d)^\vee) = \L(\calf_{d-1})Ch_{\L(\calf_{d-1,v})}(H^1(\G^d_{d-1,w},A(\calf_{d,w})[p^\infty]^\vee) \ .
\end{equation}
We also saw that, for a ramified prime $v$, $\calh_v(\calf_d)^\vee$ (which is $H^1(\G^d_{d-1,w},A(\calf_{d,w})[p^\infty]^\vee$,
because $v$ is totally ramified) is finitely generated over $\Z_p$, hence pseudo-null over $\L(\calf_{d-1,v})=\L(\calf_{d-1})$ for $d\geq 3$.

\noindent Now we check the unramified primes in $S_A\,$. If $v$ splits completely in $\calf/F$, then
$\L(\calf_{d-1,v})\simeq \Z_p$ for any $d$ and, since (by Lemma \ref{KercTor}) $H^1(\G^d_{d-1,w},A(\calf_{d,w})[p^\infty]^\vee$ is finite,
we get
\[ Ch_{\L(\calf_{d-1,v})}(H^1(\G^d_{d-1,w},A(\calf_{d,w})[p^\infty]^\vee) = |H^1(\G^d_{d-1,w},A(\calf_{d,w})[p^\infty]^\vee|=(p^{\nu_v}) \ .\]
If $v$ is inert in some extension $\calf_d/\calf_{d-1}\,$, then
\[ \L(\calf_{d-1,v})\simeq \Z_p\quad {\rm  and}\quad  \L(\calf_{r,v})\simeq \Z_p[[t_d]] \ {\rm for \ any}\ r\geq d\ .\]
Hence
\[ Ch_{\L(\calf_{r-1,v})}(H^1(\G^d_{d-1,w},A(\calf_{d,w})[p^\infty]^\vee) = (1) \ {\rm for \ any}\ r\geq d+1\ .\]
These local informations and \eqref{TanChId} yield the theorem.
\end{proof}

Now we deal with the other extra term of equation \eqref{CharIdSel2}, i.e.,
$Ch_{\L(\calf_{d-1})}(\cals(\calf_d)^{\G^d_{d-1}})$. Note first that, taking duals
\[ (\cals(\calf_d)^{\G^d_{d-1}})^\vee \simeq \cals(\calf_d)^\vee/(\g_d-1) = Sel(\calf_d)/(\g_d-1) \ ,\]
so we work on the last module.

\noindent From now on we put $\g:=\g_d$ and we shall need the following (see also \cite[Proposition 4.3.2]{Tan11})

\begin{lem}\label{Surjgamma-1}
We have
\[ H^1_{fl}(\calx_d,A[p^\infty]) = (\g-1)H^1_{fl}(\calx_d,A[p^\infty])\ .\]
\end{lem}

\begin{proof} Since
\[ H^1_{fl}(\calx_d,A[p^\infty]) = \dl{K\subset \calf_d\ [K:F]<\infty} \dl{m} H^1_{fl}(X_K,A[p^m]) \  ,\]
an element $\alpha\in H^1_{fl}(\calx_d,A[p^\infty])$ belongs to some $H^1_{fl}(X_K,A[p^m])$.
Now let $\g^{p^{s(K)}}$ be the largest power of $\g$ which acts trivially on $K$, and define a $\Z_p$-extension $K_\infty$ with
$\Gal(K_\infty/K)=\overline{\langle \g^{p^{s(K)}}\rangle}$ and layers $K_n\,$.
Take $t\geq m$, consider the restrictions
\[ H^1_{fl}(X_K,A[p^m]) \rightarrow H^1_{fl}(X_{K_t},A[p^m]) \rightarrow H^1_{fl}(\calx_{K_\infty},A[p^m]) \]
and denote by $x_t$ the image of $x$. Now $x_t$ is fixed by $\Gal(K_t/K)$ and $p^mx_t=0$, so
$x_t$ is in the kernel of the norm $N^{K_t}_K\,$, i.e., $x_t$ belongs to the (Galois) cohomology group
\[ H^1(K_t/K, H^1_{fl}(\calx_{K_\infty},A[p^m]) \iri H^1(K_\infty/K, H^1_{fl}(\calx_{K_\infty},A[p^m])\ .\]
Let $Ker^2_m$ be the kernel of the restriction map
$H^2_{fl}(X_K,A[p^m])\rightarrow H^2_{fl}(\calx_{K_\infty},A[p^m])$, then, from the Hochschild-Serre spectral sequence,
we have
\begin{equation}\label{HochSerKer}
Ker^2_m \rightarrow H^1(K_\infty/K, H^1_{fl}(\calx_{K_\infty},A[p^m]) \rightarrow H^3(K_\infty/K,A(K_\infty)[p^m])=0
\end{equation}
(because the $p$-cohomological dimension of $\Z_p$ is 1). To get rid of $Ker^2_m$ note that,
by \cite[Lemma 3.3]{GAT12}, $H^2_{fl}(X_K,A)=0$. Hence, the cohomology sequence arising from
\[ \xymatrix{ A[p^m] \ar@{^(->}[r] & A  \ar@{->>}^{p^m}[r] & A \ ,} \]
yields an isomorphism $H^2_{fl}(X_K,A[p^m])\simeq H^1_{fl}(X_K,A)/p^m\,$. Consider the commutative diagram (with $m_2\geq m_1\,$)
\[ \xymatrix { H^1_{fl}(X_K,A)/p^{m_1} \ar[r]^\sim \ar[d]_{p^{m_2-m_1}} & H^2_{fl}(X_K,A[p^{m_1}]) \ar[d] \\
H^1_{fl}(X_K,A)/p^{m_2} \ar[r]^\sim & H^2_{fl}(X_K,A[p^{m_2}]) \ .} \]
An element of $H^1_{fl}(X_K,A)/p^{m_1}$ of order $p^r$ goes to zero via the vertical map on the left as soon as
$m_2\geq m_1+r$, hence the direct limit provides ${\displaystyle\dl{m} H^1_{fl}(X_K,A)/p^m =0}$ and, eventually,
${\displaystyle\dl{m} Ker^2_m=0}$ as well. By \eqref{HochSerKer}
\[ 0= \dl{m} H^1(K_\infty/K, H^1_{fl}(\calx_{K_\infty},A[p^m]) = H^1(K_\infty/K, H^1_{fl}(\calx_{K_\infty},A[p^\infty])\ , \]
which yields
\[ H^1_{fl}(\calx_{K_\infty},A[p^\infty])=(\g^{p^{s(K)}}-1)H^1_{fl}(\calx_{K_\infty},A[p^\infty])=
(\g-1)H^1_{fl}(\calx_{K_\infty},A[p^\infty])\ .\]
We get the claim by taking the direct limit on the finite subextensions $K$.
\end{proof}

\begin{thm}\label{PNullThm}
For any $d \geq 3$ we have $Ch_{\L(\calf_{d-1})}( \cals(\calf_d)^{\G^d_{d-1}})=(1)$.
\end{thm}

\begin{proof} Consider the following diagram
\begin{equation}\label{DiagTan}
\xymatrix { Sel(\calf_d) \ar@{^(->}[r] \ar[d]^{\g-1} &
H^1_{fl}(\calx_d,A[p^\infty]) \ar[r]^{\phi} \ar@{->>}[d]^{\g-1} &
\calh^1(\calx_d,A) \ar@{->>}[r] \ar[d]^{\g-1} &
Coker(\phi) \ar[d]^{\g-1} \\
Sel(\calf_d) \ar@{^(->}[r] &
H^1_{fl}(\calx_d,A[p^\infty]) \ar[r]^{\phi}  &
\calh^1(\calx_d,A) \ar@{->>}[r]  &
Coker(\phi)  } \end{equation}
(where $\calh^i(\calx_d,A):=\displaystyle{\prod_w H^i_{fl}(\calx_{d,w},A)[p^\infty]}$ and the surjectivity of the second vertical
arrow comes from the previous lemma). Inserting $\calg(\calf_d):=Im(\phi)$, we get two diagrams
\begin{equation}\label{DiagTan1}
\xymatrix { Sel(\calf_d) \ar@{^(->}[r] \ar[d]^{\g-1} &
H^1_{fl}(\calx_d,A[p^\infty]) \ar@{->>}[r]^{\qquad\phi} \ar@{->>}[d]^{\g-1} &
\calg(\calf_d) \ar@{->>}[d]^{\g-1} \\
Sel(\calf_d) \ar@{^(->}[r] &
H^1_{fl}(\calx_d,A[p^\infty]) \ar@{->>}[r]^{\qquad\phi}  &
\calg(\calf_d) }
\xymatrix { \calg(\calf_d) \ar@{^(->}[r] \ar@{->>}[d]^{\g-1} &
\calh^1(\calx_d,A) \ar@{->>}[r] \ar[d]^{\g-1} &
Coker(\phi) \ar[d]^{\g-1} \\
\calg(\calf_d) \ar@{^(->}[r] &
\calh^1(\calx_d,A) \ar@{->>}[r]  &
Coker(\phi)  \ .} \end{equation}
From the snake lemma sequence of the first one, we obtain the isomorphism
\begin{equation}\label{QuotSelIso}
\mathcal{G}(\calf_d)^{\G^d_{d-1}}/Im(\phi^{\G^d_{d-1}}) \simeq Sel(\calf_d)/(\g-1)
\end{equation}
(where $\phi^{\G^d_{d-1}}$ is the restriction of $\phi$ to $H^1_{fl}(\calx_d,A[p^\infty])^{\G^d_{d-1}}\,$).
The snake lemma sequence of the second diagram (its ``upper'' row) yields an isomorphism
\begin{equation}\label{CokerphiIso}
\calh^1(\calx_d,A)^{\G^d_{d-1}}/\mathcal{G}(\calf_d)^{\G^d_{d-1}} \simeq Coker(\phi)^{\G^d_{d-1}} \ .
\end{equation}
The injection $\mathcal{G}(\calf_d)^{\G^d_{d-1}} \iri \calh^1(\calx_d,A)^{\G^d_{d-1}}$ induces an exact sequence
\[ \mathcal{G}(\calf_d)^{\G^d_{d-1}}/Im(\phi^{\G^d_{d-1}}) \iri \calh^1(\calx_d,A)^{\G^d_{d-1}}/Im(\phi^{\G^d_{d-1}})
\sri \calh^1(\calx_d,A)^{\G^d_{d-1}}/\mathcal{G}(\calf_d)^{\G^d_{d-1}}  \]
(with a little abuse of notation we are considering $Im(\phi^{\G^d_{d-1}})$ as a submodule of
$\calh^1(\calx_d,A)^{\G^d_{d-1}}$ via the natural injection above)
which, by \eqref{QuotSelIso} and \eqref{CokerphiIso}, yields the sequence
\begin{equation}\label{Tan42}
Sel(\calf_d)/(\g-1) \iri Coker(\phi^{\G^d_{d-1}}) \sri Coker(\phi)^{\G^d_{d-1}} \ .
\end{equation}
\noindent Now consider the following diagram
\[ \xymatrix {
H^1(\G^d_{d-1},A[p^\infty]) \ar@{^(->}[r] \ar[d]^{\phi^d_{d-1}} &
H^1_{fl}(\calx_{d-1},A[p^\infty]) \ar[r] \ar[d]^{\phi_{d-1}} &
H^1_{fl}(\calx_d,A[p^\infty])^{\G^d_{d-1}} \ar[r] \ar[d]^{\phi^{\G^d_{d-1}}} & 0 \ar[d]\\
\calh^1(\G^d_{d-1},A) \ar@{^(->}[r] &
\calh^1(\calx_{d-1},A) \ar[r] &
\calh^1(\calx_d,A)^{\G^d_{d-1}} \ar@{->>}[r] &
\calh^2(\G^d_{d-1},A)  } \]
where:\begin{itemize}
\item the vertical maps are all induced by the product of restrictions $\phi$;
\item the horizontal lines are just the Hochschild-Serre sequences for global and local cohomology;
\item the 0 in the upper right corner comes from $H^2(\G^d_{d-1},A[p^\infty]) = 0$;
\item the surjectivity on the lower right corner comes from $\calh^2(\calx_{d-1},A) = 0$, which is
a direct consequence of \cite[Theorem III.7.8]{Mi}.
\end{itemize}

\noindent This yields a sequence (from the snake lemma)
\begin{equation} \label{EqCoker} Coker(\phi_{d-1}) \rightarrow Coker(\phi^{\G^d_{d-1}}) \rightarrow  \calh^2(\G^d_{d-1},A)=
\prod_w H^2(\G^d_{d-1,w},A(\calf_{d,w}))[p^\infty] \ .\end{equation}

\noindent{\bf The module $Coker(\phi_{d-1})$.}
The Kummer map induces a surjection $H^1(\calx_{d-1},A[p^\infty])\sri H^1(\calx_{d-1},A)[p^\infty]$
which fits in the diagram
\[ \xymatrix{ H^1(\calx_{d-1},A[p^\infty]) \ar@{->>}[d] \ar[r]^{\quad\phi_{d-1}} & \calh^1(\calx_{d-1},A) \\
H^1(\calx_{d-1},A)[p^\infty] \ar[ur]_{\lambda_{d-1}} & & \ .}  \]
This induces natural surjective maps $Im(\phi_{d-1})\sri Im(\lambda_{d-1})$ and, eventually,
$Coker(\lambda_{d-1})\sri Coker(\phi_{d-1})$.
For any finite extension $K/F$ we have a similar map
\[ \lambda_K : H^1(X_K,A)[p^\infty] \rightarrow \calh^1(X_K,A) \]
whose cokernel verifies
\[ Coker(\lambda_K)^\vee \simeq T_p(Sel_{A^t}(K)_p) \]
(by \cite[Main Theorem]{GAT07}), where $A^t$ is the dual abelian variety of $A$. Moreover there is an embedding
of $T_p(Sel_{A^t}(K)_p)$ into the $p$-adic completion of $\calh^0(X_K,A^t)$ (recall that, by Tate local duality,
$A^t(K_v)=H^0(K_v,A^t)$ is the Pontrjagin dual of $H^1(K_v,A)$, see \cite[Theorem III.7.8]{Mi}).
Taking limits on all the finite subextensions of $\calf_{d-1}$ we find similar relations
\[ Coker(\lambda_{d-1})^\vee \simeq T_p(Sel_{A^t}(\calf_{d-1})_p) \iri \il{K}\il{n} \calh^0(X_K,A^t)/p^n
= \il{K}\il{n} A^t(K)[p^\infty]/p^n \ .\]
Hence $Coker(\lambda_{d-1})^\vee$ embeds into a finitely generated $\Z_p$-module, i.e., it is
$\L(\calf_{d-1})$ pseudo-null for any $d\geq 3$.

\noindent{\bf The modules $H^2(\G^d_{d-1,w},A(\calf_{d,w}))[p^\infty]$.}
If the prime splits completely in $\calf_d/\calf_{d-1}\,$, then obviously $H^2(\G^d_{d-1,w},A(\calf_{d,w}))[p^\infty]=0$.
If the place is ramified or inert, then $\G^d_{d-1,w} \simeq \Z_p\,$.
Consider the exact sequence
\[ \xymatrix{ A(\calf_{d,w})[p]\, \ar@{^(->}[r]  & A(\calf_{d,w}) \ar@{->>}[r]^{p}  & pA(\calf_{d,w}) } \ ,\]
which yields a surjection
\[ \xymatrix{ H^2(\G^d_{d-1,w},A(\calf_{d,w})[p]) \ar@{->>}[r]  & H^2(\G^d_{d-1,w},A(\calf_{d,w}))[p] \ .} \]
The module on the left is trivial because $cd_p(\Z_p)=1$, hence $H^2(\G^d_{d-1,w},A(\calf_{d,w}))[p]=0$
and this yields $H^2(\G^d_{d-1,w},A(\calf_{d,w}))[p^\infty]=0$ as well.

\noindent The sequence \eqref{EqCoker} implies that $Coker(\phi^{\G^d_{d-1}})$ is $\L(\calf_{d-1})$ pseudo-null for $d\geq 3$
and, by \eqref{Tan42}, we get $Sel(\calf_d)/(\g-1)$ is pseudo-null as well. Therefore
\[ Ch_{\L(\calf_{d-1})}( \cals(\calf_d)^{\G^d_{d-1}}) = Ch_{\L(\calf_{d-1})}((Sel(\calf_d)/(\g-1))^\vee) = (1)\ .\]
\end{proof}

A direct consequence of equation \eqref{CharIdSel2} and Theorems \ref{CharIdCoker} and \ref{PNullThm} is

\begin{cor}\label{CorTan} Assume $\calf/F$ is a totally ramified extension, then, for any $d\gg 0$
and any $\Z_p$-subextension $\calf_d/\calf_{d-1}\,$, one has
\begin{equation}\label{CharIdSel5}
\pi^d_{d-1} (Ch_{\L(\calf_d)}(\cals(\calf_d))) =
(p^\nu)\cdot Ch_{\L(\calf_{d-1})}(\cals(\calf_{d-1})) \ .
\end{equation}
Moreover if no unramified places $v\in S_A$ splits completely in $\calf$, then
\begin{equation}\label{CharIdSel4}
\pi^d_{d-1} (Ch_{\L(\calf_d)}(\cals(\calf_d))) = Ch_{\L(\calf_{d-1})}(\cals(\calf_{d-1})) \ .
\end{equation}
\end{cor}

\noindent In any case the modules $\cals(\calf_d)$ verify the hypotheses of Theorem \ref{IntroThmBBL2} (because
of Proposition \ref{FinGenSel} and Theorem \ref{PNullThm}), so we can define

\begin{df}\label{CharIdSel}
The \emph{pro-characteristic ideal} of $\cals(\calf)$ is
\[ \wt{Ch}_{\L(\calf)}(\cals(\calf)):=
\il{d}\, (\pi^{\L(\calf)}_{\L(\calf_d)})^{-1}(Ch_{\Lambda(\calf_d)}(\cals(\calf_d)))\ \]
where $\pi^{\L(\calf)}_{\L(\calf_d)}:\L(\calf)\rightarrow\Lambda(\calf_d)$ is the natural projection map.
\end{df}

\noindent We remark that Definition \ref{CharIdSel} only depends on the extension $\calf/F$ and not on the filtration of
$\Z_p^d$-extension we choose inside it. Indeed with two different filtrations $\{\calf_d\,\}$ and $\{\calf'_d\,\}$
we can define a third one by putting
\[ \calf''_0:=F\quad{\rm and}\quad \calf''_n=\calf_n\calf'_n \quad \forall\,n\geq 1 \ .\]
By Corollary \ref{CorTan}, the limits of the characteristic ideals of the filtrations we started with coincide
with the limit on the filtration $\{\calf''_n\,\}$ (see \cite[Remark 3.11]{BBL2} for an analogous statement
for characteristic ideals of class groups).

This pro-characteristic ideal could play a role in the Iwasawa Main Conjecture (IMC) for a
totally ramified extension of $F$ as the algebraic counterpart of a $p$-adic $L$-function associated to $A$ and $\calf$
(see \cite[Section 5]{BL} or \cite[Section 3]{BBL} for similar statements but with Fitting ideals).
Anyway, at present, the problem of formulating a (conjectural) description of this ideal in terms of a natural
$p$-adic $L$-functions (i.e., a general non-noetherian Iwasawa Main Conjecture) is still wide open.
However, we can say something if $A$ is already defined over the constant field of $F$.

\begin{thm}\label{FinalThm}{\bf [Non-noetherian IMC for constant abelian varieties]}
Assume $A/F$ is a constant abelian variety and let $\calf/F$ be a totally ramified extension as above. Then there exists an element
$\theta_{A,\calf}$ interpolating the classical $L$-function $L(A,\chi,1)$ (where $\chi$ varies among characters of $\Gal(\calf/F)$)
such that one has an equality of ideals in $\L(\calf)$
\begin{equation} \label{e:lltt}
\wt{Ch}_{\L(\calf)}(\cals(\calf)) = (\theta_{A,\calf})\ . \end{equation}
\end{thm}

\begin{proof}
This is a simple consequence of \cite[Theorem 1.3]{LLTT}. Namely, the element $\theta_{A,L}$ is defined in \cite[Section 7.2.1]{LLTT}
for any abelian extension $L/F$ unramified outside a finite set of places. It satisfies
$\pi^d_{d-1}(\theta_{A,\calf_d})=\theta_{A,\calf_{d-1}}$ by construction and the interpolation formula (too complicated
to report it here) is proved in \cite[Theorem 7.3.1]{LLTT}. Since $A$ has good reduction everywhere, \eqref{CharIdSel5}
always holds, so both sides of \eqref{e:lltt} are defined. Finally \cite[Theorem 1.3]{LLTT} proves that
$Ch_{\L(\calf_d)}(\cals(\calf_d)) = (\theta_{A,\calf_d})$ for all $d$ and \eqref{e:lltt} follows by just taking a limit.
\end{proof}

\end{document}